\documentclass{amsart}

\usepackage[latin1]{inputenc}
\usepackage{comment}
\usepackage{qtree,array,youngtab}

\usepackage{color, graphicx}
\usepackage{soul}
\usepackage[T1]{fontenc}

\definecolor{darkgreen}{rgb}{0.,0.5,0.}

 \newtheorem{thm}{Theorem}[section]
 \newtheorem{cor}[thm]{Corollary}
 \newtheorem{lem}[thm]{Lemma}
 
\title{The number of simultaneous core partitions}
\author{Huan Xiong}
\address{I-Math, Universit$\ddot{a}$t
Z$\ddot{u}$rich, Winterthurerstrasse 190, Z$\ddot{u}$rich 8057,
Switzerland}
\email{huan.xiong@math.uzh.ch}

\subjclass{05A17, 11P81}

\keywords{partition, hook length, $\beta$-set, $t$-core}

\begin{document}
\begin{abstract}
Amdeberhan conjectured that the number of $(t,t+1, t+2)$-core
partitions is $\sum_{0\leq k\leq [\frac{t}{2}]
}\frac{1}{k+1}\binom{t}{2k}\binom{2k}{k}$. In this paper, we obtain
the generating function of the  numbers $f_t$ of $(t, t + 1, \ldots,
t + p)$-core partitions. In particular, this verifies that
Amdeberhan's conjecture is true. We also prove that the number of
$(t_1,t_2,\ldots, t_m)$-core partitions is finite if and only if
gcd$(t_1,t_2,\ldots, t_m)=1,$ which extends Anderson's result on the
finiteness of the number of $(t_1,t_2)$-core partitions for coprime
positive integers  $t_1$ and $t_2$ and thus rediscover a result of
Keith and Nath with a different proof.
\end{abstract}

\maketitle

\section{Introduction} Partitions  of positive integers are widely studied in number theory
and combinatorics. As we know, a \emph{partition} of  a positive
integer $n$ is a finite non increasing sequence of positive integers
$\lambda = (\lambda_1, \lambda_2, \ldots, \lambda_r)$ with
$\sum_{1\leq i\leq r}\lambda_i=n$. In this case, $n$ is called the
\emph{size} of $\lambda$, which is also be denoted by $\mid \lambda
\mid$. We can associate a partition $\lambda$ with its \emph{Young
diagram}, which is an array  of boxes arranged in left-justified
rows with $\lambda_i$ boxes in the $i$-th row. To the $(i,j)$-box of
the Young diagram, let $h(i, j)$ be its \emph{hook length}, which is
the number of boxes directly to the right, directly below, or the
box itself. Let $t$ be a positive integer.  A partition $\lambda$ is
called a \emph{$t$-core partition} if none of its hook lengths is a
multiple of $t$. Finally, we say that $\lambda$ is  a
\emph{$(t_1,t_2,\ldots, t_m)$}-core partition if it is
simultaneously  a $t_1$-core, a $t_2$-core, $\ldots$, a $t_m$-core
partition. For instance, Figure $1$ shows the Young diagram and hook
lengths of the partition $(5,2,2)$. It is easy to see that, the
partition $(5,2,2)$ is a $(4, 5)$-core partition since non of its
hook lengths is divisible by $4$ or $5$.

\begin{figure}[htbp]
\begin{center}
\Yvcentermath1

\begin{tabular}{c}
$\young(76321,32,21)$

\end{tabular}

\end{center}
\caption{The Young diagram and hook lengths of the partition
$(5,2,2)$.}
\end{figure}

 For  $t$-core partitions,  Granville and Ono \cite{gran} proved that there
always exists a $t$-core partition with size $n$  for any $t\geq 4$
and $n\geq 1$. A very important result in the study of
$(t_1,t_2,\ldots, t_m)$-core partitions was given by Anderson
\cite{and}, that is, there are only finite $(t_1,t_2)$-core
partitions when $t_1$ and $t_2$ are coprime to each other. Actually,
Anderson showed that the number of $(t_1,t_2)$-core partitions is
exactly $\frac{1}{t_1+t_2} \binom{t_1+t_2}{t_1}$ for relatively
prime positive integers $t_1$ and $t_2$.  Anderson's beautiful
result attracts much attention and motives a lot of work in the
study of simultaneous core partitions. Stanley and Zanello
\cite{stanley} showed that the average size of a $(t,t+1)$-core
partition is $\binom{t+1}{3}/2.$ In $2007$, Olsson and Stanton
\cite{ols} proved that the largest size of $(t_1,t_2)$-core
partitions is $ \frac{({t_1}^2-1)({t_2}^2-1)}{24}$ when $t_1$ and
$t_2$ are coprime to each other. Ford, Mai, and Sze \cite{ford}
showed that the number of self-conjugate $(t_1,t_2)$-core partitions
is $\binom{[\frac{t_1}{2}]+[\frac{t_2}{2}]}{[\frac{t_1}{2}]}$ for
relatively prime positive integers $t_1$ and $t_2$, where $[ x ]$
denotes the largest integer not greater than $x$.

Anderson \cite{and} proved  the finiteness of the number of
$(t_1,t_2)$-core partitions for coprime positive integers  $t_1$ and
$t_2$. We will extend Anderson's this result to a more general case
and thus rediscover Theorem $1$ in \cite{KN} with a different proof:
\begin{thm} \label{main3}
The number of $(t_1,t_2,\ldots, t_m)$-core partitions is finite if
and only if gcd$(t_1,t_2,\ldots, t_m)=1,$ where gcd$(t_1,t_2,\ldots,
t_m)$ denotes the greatest common divisor of $t_1,t_2,\ldots, t_m.$
\end{thm}

For the number of  $(t,t+1, t+2)$-core partitions, Amdeberhan
\cite{tamd} gave the following conjecture, which we will  prove in
Section \textbf{$3$}:

\begin{thm} \label{main2}
(Cf. Conjecture 11.1 of  \cite{tamd}.)   The number $f_t$ of $(t, t
+ 1, t + 2)$-core partitions is the $t-$th Motzkin number
$\sum_{0\leq k\leq [\frac{t}{2}]
}\frac{1}{k+1}\binom{t}{2k}\binom{2k}{k}$. The generating function
of $f_t$ is
$$\sum \limits_{t\geq 0}f_tx^t= \frac{1-x-\sqrt{1-2x-3x^2}}{2x^2}.$$
\end{thm}

\section{Proof of Theorem \ref{main3}}

Suppose that $\lambda_1 \geq \lambda_2\geq \cdots\geq \lambda_r\geq
1$ and $\lambda = (\lambda_1, \lambda_2, \ldots, \lambda_r)$ is a
partition. The \emph{$\beta$-set} of $\lambda$ is
 denoted by
$$\beta(\lambda)=\{\lambda_i+r-i : 1 \leq i \leq r\}.$$ It is obvious that $0\notin\beta(\lambda).$ Actually $\beta(\lambda)$
 is just the set of
hook lengths of boxes in the first column of the corresponding Young
diagram. It is easy to see that a partition $\lambda$ is uniquely
determined by its $\beta$-set $\beta(\lambda)$.  The following is a
well-known result on  $\beta$-sets of $t$-core partitions.






\begin{lem} \label{a-t} (\cite{jk}.)
A partition $\lambda$ is a $t$-core partition if and only if for any
$x\in \beta(\lambda)$  such that $x\geq t$, we have $x-t \in
\beta(\lambda)$.
\end{lem}

By Lemma \ref{a-t}, we can easily deduce the following result:
\begin{lem} \label{linear}
Let $\lambda$ be a $(t_1,t_2,\ldots, t_m)$-core partition and $a_i$
be some non negative  integers. Then $\sum_{1\leq i\leq
m}a_it_i\notin \beta(\lambda).$
\end{lem}
\begin{proof} We will prove this result by induction on $m$. If
$m=1$, by Lemma \ref{a-t} we know $a_1t_1\notin \beta(\lambda)$
since $0\notin \beta(\lambda)$. Now we assume that $m\geq 2$ and the
result is true for $m-1$, i.e., $\sum_{1\leq i\leq m-1}a_it_i\notin
\beta(\lambda)$ if  $a_i$ are some non negative  integers. Then by
Lemma \ref{a-t} we know $\sum_{1\leq i\leq
m}a_it_i=a_mt_m+\sum_{1\leq i\leq m-1}a_it_i\notin \beta(\lambda)$.
 \end{proof}

\vspace{1ex}

 Now we can prove Theorem \ref{main3}.

\noindent\textbf{Proof of Theorem \ref{main3}.} $\Rightarrow$:
Suppose that gcd$(t_1,t_2,\ldots, t_m)=d>1.$ For every  $n\in
\textbf{N}$, let $\lambda_n$ be the partition whose $\beta$-set is
$$\beta(\lambda_n)=\{1,1+d,1+2d,\ldots, 1+nd\}.$$ Then for any $1\leq
i\leq m$ and $0\leq j\leq n$  such that $1+jd\geq t_i$, we have
$1+jd-t_i=1+j'd \in \beta(\lambda_n)$ for some non negative integer
$j'$ since $d\mid t_i$ and $d>1$. Then by Lemma \ref{a-t},
$\lambda_n$ is a $(t_1,t_2,\ldots, t_m)$-core partition for every
$n\in \textbf{N}$. This means that the number of $(t_1,t_2,\ldots,
t_m)$-core partitions is infinite.

$\Leftarrow$: Suppose that gcd$(t_1,t_2,\ldots, t_m)=1$ and $1\leq
t_1<t_2<\cdots< t_m$. To show that the number of $(t_1,t_2,\ldots,
t_m)$-core partitions is finite, we just need to show that for every
$(t_1,t_2,\ldots, t_m)$-core partition $\lambda$ and
$x\geq(t_1-1)\sum_{2\leq i\leq m}t_i$, we have $x\notin
\beta(\lambda)$:

 First we know there
exist some  $a_i\in \textbf{Z}$ such that $x=\sum_{1\leq i\leq
m}a_it_i $ since gcd$(t_1,t_2,\ldots, t_m)=1$. Furthermore, we can
assume that $0\leq a_i\leq t_1-1$ for $2\leq i\leq m$ since
$$a_1t_1+a_it_i= (a_1-bt_i)t_1+(a_i+bt_1)t_i
$$ for every $b\in \textbf{Z}$. Now we have $$\sum \limits_{1\leq i\leq m}a_it_i
=x\geq(t_1-1)\sum \limits_{2\leq i\leq m}t_i
$$ and $0\leq a_i\leq t_1-1$ for
$2\leq i\leq m$. It follows that $$a_1t_1=x-\sum \limits_{2\leq
i\leq m}a_it_i\geq x-(t_1-1)\sum \limits_{2\leq i\leq m}t_i\geq0.$$
Thus we know $a_1\geq 0$. Then by Lemma \ref{linear}, we have
$$x=\sum \limits_{1\leq i\leq m}a_it_i \notin\beta(\lambda).$$  This means that $x\notin \beta(\lambda)$ if
$\lambda$ is a $(t_1,t_2,\ldots, t_m)$-core partition and
$x\geq(t_1-1)\sum_{2\leq i\leq m}t_i$. Now we know for a
$(t_1,t_2,\ldots, t_m)$-core partition $\lambda$, its $\beta$-set
$\beta(\lambda)$ must be a subset of $\{1,2,\ldots,
(t_1-1)\sum_{2\leq i\leq m}t_i-1 \}$. This implies that the number
of $(t_1,t_2,\ldots, t_m)$-core partitions must be finite. \hfill
$\square$

\section{Main results}
 Throughout this section, let $p$ be a given positive integer.

 Let
$S_{t,i}=\{  x\in \textbf{Z} : (i-1)(t+p)+1\leq x \leq  it-1 \}.$
The following result is a characterization of $\beta$-sets of
$(t,t+1,\ldots, t+p)$-core partitions.

\begin{lem} \label{set1}
Suppose that $\lambda$ is a $(t,t+1,\ldots, t+p)$-core partition.
Then
$$\beta(\lambda)\subseteq  \bigcup_{1\leq i \leq [\frac{t+p-2}{p}]}
S_{t,i}.$$
\end{lem}
\begin{proof} By Lemma \ref{linear}, we have $\sum_{0\leq k\leq
p}a_k(t+k)\notin \beta(\lambda )$ for non negative integers $a_k.$
Let
$$T_{t,i}=\{\sum \limits_{0\leq k \leq p}{a_k(t+k)}: a_k\in
\textbf{Z},\ a_k\geq 0,\ \sum \limits_{0\leq k \leq p}{a_k}=i\}. $$
Then $T_{t,i}\bigcap \beta(\lambda )=\emptyset$ for $i\geq 0.$ It is
easy too see that
$$T_{t,i}= \{  x\in \textbf{Z}: it\leq x \leq i(t+p) \}$$ and
$$\bigcup_{i \geq [\frac{t+p-2}{p}]} T_{t,i}=\{ x\in \textbf{Z}:
x\geq [\frac{t+p-2}{p}]t   \}$$ since $(i+1)t-1\leq it+
[\frac{t+p-2}{p}]p\leq i(t+p)$ for $i\geq [\frac{t+p-2}{p}]$. Thus
$\beta(\lambda)$ must be a subset of $$\{
 x\in \textbf{Z} :1\leq x\leq [\frac{t+p-2}{p}]t-1 \} \setminus (\mathop{\bigcup}_{1\leq
 i
\leq [\frac{t+p-2}{p}]-1} {\{x\in \textbf{Z}: it\leq x \leq i(t+p)
\}}),$$ which equals to $\bigcup_{1\leq i \leq [\frac{t+p-2}{p}]}
S_{t,i}$. \end{proof}

\vspace{1ex}

 We can define a partial order relation on $\bigcup_{1\leq i
\leq [\frac{t+p-2}{p}]} S_{t,i}$. That is, for every $x,y\in
\bigcup_{1\leq i \leq [\frac{t+p-2}{p}]} S_{t,i}$, we define
$y\preceq x$ if and only if $x-y=\sum_{0\leq k\leq p}a_k( t+k)$ for
some non negative integers $a_k.$ It is easy to verify that
$\preceq$ is indeed a partial order relation. We say that a subset
$S$ of a partially ordered set $T$ is \emph{good}
 if for every $x\in S$, $y\in T$ such that $y\preceq x$ in $T$, we always
 have $y\in S$.

By the definition of $S_{t,i}$, It is easy to see that
$$S_{t,i}=\{x-(t+k):x\in S_{t,i+1},\ 0\leq k\leq p\}$$ for $1\leq
i\leq [\frac{t+p-2}{p}]-1$. Then by Lemma \ref{a-t} and Lemma
\ref{set1} the following result is obvious:

\begin{lem} \label{num}
A partition $\lambda$ is a $(t, t + 1, \ldots, t + p)$-core
partition if and only if $\beta(\lambda)$ is a good subset of $\
\bigcup_{1\leq i \leq [\frac{t+p-2}{p}]} S_{t,i}.$
\end{lem}

 Let $R_{t,j}$ be the set of $ (t, t +
1, \ldots, t + p)$-core partitions whose
 $\beta$-sets contain every positive integer smaller than $j$ but don't contain
 $j$. Let $r_{t,j}=\#R_{t,j}$ be the number of elements in
 $R_{t,j}$.

 Now we can give the main result in this paper.

\begin{thm} \label{main}
Suppose that $p$ is a given positive integer. The number $f_t$ of
$(t, t + 1, \ldots, t + p)$-core partitions is computed recursively
by
$$f_t=0\ \text{for}\ t<0;\ f_0 = 1; \  f_t = \sum_{ i=1}^{p-1}f_{t-i} + \sum_{j=0}^{t-p}
f_jf_{t-p-j} \ \text{for}\ t\geq 1.$$ The generating function of
$f_t$ is $$\sum \limits_{t\geq 0}f_tx^t=\frac{1-\sum_{1\leq i\leq
p-1}x^i-\sqrt{(1-\sum_{1\leq i\leq p-1}x^i)^2-4x^p}}{2x^p}.$$
\end{thm}
\begin{proof} For convenience, let $f_t=0\ \text{for}\
t<0$ and $ f_0 = 1$. Now  suppose that $t\geq 1.$ First we know
$r_{t,j}=0$ for $j\geq t+1$  since $t\notin \beta(\lambda)$ and thus
$f_t =\sum_{1\leq j\leq t}r_{t,j}$.

\textbf{Step 1.}   We claim that $r_{t,j}=f_{t-j}$ for $1\leq j\leq
p-1:$

Notice that $r_{t,j}=f_{t-j}=0$ is true if $t+1\leq j\leq p-1$ since
we already assume that $f_t=0\ \text{for}\ t<0$. Now we can assume
that $1\leq j\leq p-1$ and $j\leq t$. Let $\lambda $ be a partition
such that $1,2,\ldots,j-1\in \beta(\lambda)$ and $j\notin
\beta(\lambda).$ If $\lambda\in R_{t,j}$, i.e., $\lambda$ is a $(t,
t + 1, \ldots, t + p)$-core partition, then by Lemma \ref{a-t}, we
have  $ x \notin \beta(\lambda)$ for $i\geq 2$ and $
(i-1)(t+p)+1\leq x \leq (i-1)(t+p)+j $ since $j\notin
\beta(\lambda)$ and $t\leq t+p+1-j\leq t+p$. Let
\begin{eqnarray*}S'_{t,i}&=&       S_{t,i}\setminus \{x\in \textbf{Z}:   (i-1)(t+p)+1\leq x \leq (i-1)(t+p)+j  \}           \\
&=& \{ x\in \textbf{Z} : (i-1)(t+p)+j+1\leq x \leq  it-1
\}.\end{eqnarray*}   Notice that $S'_{t,i}=\emptyset$ when $i>
[\frac{t-j+p-2}{p}]$. Thus it is easy to see that
$$\{ 1,2,\ldots,j-1\}\subseteq \beta(\lambda)\subseteq (\bigcup_{1\leq i \leq [\frac{t-j+p-2}{p}]} S'_{t,i})\bigcup \{
1,2,\ldots,j-1\}$$ if $\lambda\in R_{t,j}$. We can define a  partial
order relation  on $\bigcup_{1\leq i \leq [\frac{t-j+p-2}{p}]}
S'_{t,i}$ induced by the partial order relation $\preceq$ on
$\bigcup_{1\leq i \leq [\frac{t+p-2}{p}]} S_{t,i}$.  That is, for
every two integers $x,y$ in $\bigcup_{1\leq i \leq
[\frac{t-j+p-2}{p}]} S'_{t,i}$, we have $y\preceq x$ if and only if
$x-y=\sum_{0\leq k\leq p}a_k( t+k)$ for some non negative integers
$a_k.$

Let $\lambda'$ be a partition such that
$$\beta(\lambda')=\beta(\lambda)\setminus \{ 1,2,\ldots,j-1\}.$$ By the
 definition of $S'_{t,i}$, we know for $1\leq
i\leq [\frac{t-j+p-2}{p}]-1$,
$$S'_{t,i}=\{x-(t+k):x\in S'_{t,i+1},\ 0\leq k\leq p\}.$$   Then by Lemma \ref{a-t}, it is easy
to see that $\lambda\in R_{t,j}$ if and only if $\lambda'$ is a $(t,
t + 1, \ldots, t + p)$-core partition with $\beta(\lambda')\subseteq
\bigcup_{1\leq i \leq [\frac{t-j+p-2}{p}]} S'_{t,i},$ which is
equivalent to $\beta(\lambda')$ is a good subset of $\bigcup_{1\leq
i \leq [\frac{t-j+p-2}{p}]} S'_{t,i}$.

Notice that $\bigcup_{1\leq i \leq [\frac{t-j+p-2}{p}]} S_{t-j,i}$
is a partially ordered set and for every two integers $x',y'$ in
$\bigcup_{1\leq i \leq [\frac{t-j+p-2}{p}]} S_{t-j,i}$, we know
$y'\preceq x'$ if and only if $x'-y'=\sum_{0\leq k\leq p}a_k(
t-j+k)$ for some non negative integers $a_k.$ Now we can  build a
function
$$\phi: \bigcup_{1\leq i \leq [\frac{t-j+p-2}{p}]}
S'_{t,i}\rightarrow \bigcup_{1\leq i \leq [\frac{t-j+p-2}{p}]}
S_{t-j,i},$$ that is, for every $x\in S'_{t,i}$, let $\phi(x)=x-ij.$
Then it is obvious that $\phi$ is a bijection. Let $x\in S'_{t,i+1}
$ and $y\in S'_{t,i} $. We have $\phi(x)-\phi(y)=x-y-j.$ Thus we
know $t-j\leq \phi(x)-\phi(y) \leq t-j+p$   if and only if $t\leq
x-y \leq t+p$, which implies that $\phi(y)\preceq\phi(x)$ in
$\bigcup_{1\leq i \leq [\frac{t-j+p-2}{p}]} S_{t-j,i}$ if and only
if $y\preceq x$ in $\bigcup_{1\leq i \leq [\frac{t-j+p-2}{p}]}
S'_{t,i}$.  This means that $\phi$ is an isomorphism of partially
ordered sets. Then $\bigcup_{1\leq i \leq [\frac{t-j+p-2}{p}]}
S'_{t,i}$ and $\bigcup_{1\leq i \leq [\frac{t-j+p-2}{p}]} S_{t-j,i}$
has the same number of good subsets and thus   by Lemma \ref{num} we
have $r_{t,j}=f_{t-j}$. We mention that if $j=t\leq p-1$, then
$r_{t,t}=f_{0}=1$ is true  since in this case, we have
$\bigcup_{1\leq i \leq [\frac{t-j+p-2}{p}]} S'_{t,i}=\emptyset$ and
the
 empty subset of a partially ordered set is always a good subset.

\textbf{Step 2.} We claim that  $r_{t,j}=f_{j-p}f_{t-j}$ for $p\leq
j\leq t:$

 Let $\lambda\in R_{t,j}$, i.e.,  $\lambda $ is a $(t, t +
1, \ldots, t + p)$-core partition  such that $1,2,\ldots,j-1\in
\beta(\lambda)$ and $j\notin \beta(\lambda).$ If $i\geq 0$ and $
it+j\leq x \leq i(t+p)+j $, by Lemma \ref{a-t} we have  $ x \notin
\beta(\lambda)$. Let
$$S'_{t,i}=\{ x\in \textbf{Z} : i(t+p)+1\leq x \leq  it+j-1
\}$$ and  $$S''_{t,i}=\{ x\in \textbf{Z} : (i-1)(t+p)+j+1\leq x \leq
it-1 \}.$$ Then $$S'_{t,i}\bigcup S''_{t,i+1}= S_{t,i+1}\setminus
\{x\in \textbf{Z}:   it+j\leq x \leq i(t+p)+j   \}
$$ and
$$S'_{t,0}=\{ 1,2,\ldots,j-1\}\subseteq \beta(\lambda).$$ Notice that $S'_{t,i}=\emptyset$ when $i>
[\frac{j-2}{p}]$ and $S''_{t,i}=\emptyset$ when $i>
[\frac{t-j+p-2}{p}]$. Thus it is easy to see
$$\beta(\lambda)\subseteq
 (\bigcup_{1\leq i \leq [\frac{j-2}{p}]} S'_{t,i})\bigcup (\bigcup_{1\leq i \leq
[\frac{t-j+p-2}{p}]} S''_{t,i})\bigcup S'_{t,0}.$$   We can define
partial order relations  on $\bigcup_{1\leq i \leq [\frac{j-2}{p}]}
S'_{t,i}$ and $\bigcup_{1\leq i \leq [\frac{t-j+p-2}{p}]} S''_{t,i}$
induced by the partial order relation $\preceq$ on $\bigcup_{1\leq i
\leq [\frac{t+p-2}{p}]} S_{t,i}$ as in Step $1$.

Let $\lambda'$ be the partition such that
$$\beta(\lambda')=(\beta(\lambda)\bigcap (\bigcup_{1\leq i \leq
[\frac{j-2}{p}]} S'_{t,i}))\bigcup S'_{t,0}$$ and $\lambda''$ be the
partition such that $$\beta(\lambda'')=\beta(\lambda)\bigcap
(\bigcup_{1\leq i \leq [\frac{t-j+p-2}{p}]} S''_{t,i}).$$   By the
definition of $S'_{t,i}$ and $S''_{t,i}$, we know
$$S'_{t,i}=\{x-(t+k):x\in S'_{t,i+1},\ 0\leq k\leq p\}$$ for $0\leq
i\leq [\frac{j-2}{p}]-1$ and $$S''_{t,i}=\{x-(t+k):x\in
S''_{t,i+1},\ 0\leq k\leq p\}$$ for $1\leq i\leq
[\frac{t-j+p-2}{p}]-1$.
 Then by Lemma \ref{a-t} it is
easy to see that $\lambda'$ and $\lambda''$ are $(t, t + 1, \ldots,
t + p)$-core partitions since $\lambda$ is a $(t, t + 1, \ldots, t +
p)$-core partition. On the other hand, if $\lambda'$ and $\lambda''$
are $(t, t + 1, \ldots, t + p)$-core partitions such that
$$S'_{t,0}\subseteq \beta(\lambda')\subseteq (\bigcup_{1\leq i \leq [\frac{j-2}{p}]}
S'_{t,i})\bigcup S'_{t,0}$$ and $$\beta(\lambda'')\subseteq
\bigcup_{1\leq i \leq [\frac{t-j+p-2}{p}]} S''_{t,i},$$ by Lemma
\ref{a-t} we can reconstruct the $(t, t + 1, \ldots, t + p)$-core
partition $\lambda\in R_{t,j}$ by letting
$$\beta(\lambda)=\beta(\lambda')\bigcup\beta(\lambda''),$$
 which implies that $$(\beta(\lambda)\bigcap (\bigcup_{1\leq i \leq
[\frac{j-2}{p}]} S'_{t,i}))\bigcup S'_{t,0}=\beta(\lambda')$$ and
$$\beta(\lambda)\bigcap (\bigcup_{1\leq i \leq [\frac{t-j+p-2}{p}]}
S''_{t,i})=\beta(\lambda'').$$  Thus the number of $(t, t + 1,
\ldots, t + p)$-core partitions in $R_{t,j}$ equals to the number of
pairs $(\lambda',\ \lambda'')$ such that $\lambda'$ and $\lambda''$
are $(t, t + 1, \ldots, t + p)$-core partitions, $S'_{t,0}\subseteq
\beta(\lambda')\subseteq (\bigcup_{1\leq i \leq [\frac{j-2}{p}]}
S'_{t,i})\bigcup S'_{t,0}$, and $\beta(\lambda'')\subseteq
\bigcup_{1\leq i \leq [\frac{t-j+p-2}{p}]} S''_{t,i}$,  which equals
to the product of the number of good subsets of $\bigcup_{1\leq i
\leq [\frac{j-2}{p}]} S'_{t,i}$ and  the number of good subsets of
$\bigcup_{1\leq i \leq [\frac{t-j+p-2}{p}]} S''_{t,i}$  by Lemma
\ref{a-t}.

First we compute the number of good subsets of $\bigcup_{1\leq i
\leq [\frac{j-2}{p}]} S'_{t,i}.$ Notice that  for every two integers
$x',y'$ in $\bigcup_{1\leq i \leq [\frac{j-2}{p}]} S_{j-p,i}$, we
have $y'\preceq x'$ if and only if $x'-y'=\sum_{0\leq k\leq p}a_k(
j-p+k)$ for some non negative integers $a_k.$  We define a function
  $$\phi: \bigcup_{1\leq i \leq
[\frac{j-2}{p}]} S'_{t,i}\rightarrow \bigcup_{1\leq i \leq
[\frac{j-2}{p}]} S_{j-p,i}$$ such that  for every $x\in S'_{t,i}$,
let $\phi(x)=x-i(t+p-j)-j.$ Then it is easy to see that $\phi$ is a
bijection. Let $x\in S'_{t,i+1} $ and $y\in S'_{t,i} $. We have
$\phi(x)-\phi(y)=x-y-(t+p-j).$ Thus $\phi(y)\preceq\phi(x)$ if and
only if $y\preceq x$ since both of them are equivalent to $ t\leq
x-y \leq t+p$. This means that $\phi$ is an isomorphism of partially
ordered sets. Then  $\bigcup_{1\leq i \leq [\frac{j-2}{p}]}
S'_{t,i}$ and $\bigcup_{1\leq i \leq [\frac{j-2}{p}]} S_{j-p,i}$ has
the same number of good subsets, which equals to $f_{j-p}$  by Lemma
\ref{num}. We mention that $\bigcup_{1\leq i \leq [\frac{j-2}{p}]}
S'_{t,i}$ has $f_{j-p}$  good subsets is  true for $j-p=0$ since the
 empty subset of a partially ordered set is always a good subset and we already assume that
 $f_0=1$.

Next we compute the number of good subsets of $\bigcup_{1\leq i \leq
[\frac{t-j+p-2}{p}]} S''_{t,i}.$    Notice that  for every two
integers $x',y'$ in $\bigcup_{1\leq i \leq [\frac{t-j+p-2}{p}]}
S_{t-j,i}$, we have $y'\preceq x'$ if and only if $x'-y'=\sum_{0\leq
k\leq p}a_k( t-j+k)$ for some non negative integers $a_k.$      We
define a function
  $$\varphi: \bigcup_{1\leq i \leq
[\frac{t-j+p-2}{p}]} S''_{t,i}\rightarrow \bigcup_{1\leq i \leq
[\frac{t-j+p-2}{p}]} S_{t-j,i}$$ such that  for every $x\in
S''_{t,i}$, let $\varphi(x)=x-ij.$ Then it is easy to see that
$\varphi$ is a bijection. Let $x\in S''_{t,i+1} $ and $y\in
S''_{t,i} $. We have $\varphi(x)-\varphi(y)=x-y-j.$ Thus
$\varphi(y)\preceq\varphi(x)$ if and only if $y\preceq x$ since both
of them are equivalent to $ t\leq x-y \leq t+p$. This means that
$\varphi$ is an isomorphism of partially ordered sets. Then
$\bigcup_{1\leq i \leq [\frac{t-j+p-2}{p}]} S''_{t,i}$ and
$\bigcup_{1\leq i \leq [\frac{t-j+p-2}{p}]} S_{t-j,i}$ has the same
number of good subsets, which equals to $f_{t-j}$  by Lemma
\ref{num}. We mention that $\bigcup_{1\leq i \leq
[\frac{t-j+p-2}{p}]} S''_{t,i}$ has $f_{t-j}$  good subsets is  true
for $t-j=0$ since the
 empty subset of a partially ordered set is always a good subset and we already assume that
 $f_0=1$.

Now we have $r_{t,j}=f_{j-p}f_{t-j}$ for $p\leq j\leq t$  by Lemma
\ref{num} and prove the claim.

\textbf{Step 3.} Put Step 1 and Step 2 together, we have $$f_t
=\sum_{j=1}^{t}r_{t,j}=\sum_{ i=1}^{ p-1}f_{t-i} + \sum_{j=p}^{t}
f_{j-p}f_{t-j}=\sum_{ i=1}^{ p-1}f_{t-i} + \sum_{j=0}^{t-p}
f_jf_{t-p-j}$$ for $t\geq 1$.

 Let $F(x)=\sum_{t\geq 0}f_tx^t$ be the generating
function of $f_t$. Then  we have
\begin{eqnarray*}
F(x)-1&=&\sum_{t\geq 1}f_tx^t=\sum_{t\geq 1}(\sum_{ i=1}^{
p-1}f_{t-i} + \sum_{j=0}^{t-p} f_jf_{t-p-j})x^t\\&=& \sum_{ i=1}^{
p-1}\sum_{t\geq 1}f_{t-i}x^t + \sum_{t\geq 1}\sum_{j=0}^{t-p}
f_jf_{t-p-j}x^t\\&=& \sum_{ i=1}^{ p-1}x^iF(x)+x^p(F(x))^2.
\end{eqnarray*}

Then $F(x)=\frac{1-\sum_{1\leq i\leq p-1}x^i-\sqrt{(1-\sum_{1\leq
i\leq p-1}x^i)^2-4x^p}}{2x^p}$. We finish the proof.   \end{proof}


\vspace{1ex}

Suppose that $p=1$ in Theorem \ref{main}. We can give a new proof of
Anderson's result  on the number of $(t_1,t_2)$-core partitions in
\cite{and}  for the case $t_2=t_1+1$.

\begin{cor} \label{cor1}
  The number $f_t$ of $(t, t + 1)$-core partitions is  $f_t = \frac{1}{2t+1}\binom{2t+1}{t}$. The generating function of $f_t$ is
$$\sum \limits_{t\geq 0}f_tx^t= \frac{1-\sqrt{1-4x}}{2x}.$$
\end{cor}
\begin{proof} Let $p=1$ in Theorem \ref{main}. We have the generating
function of $f_t$ is
$$\sum \limits_{t\geq 0}f_tx^t= \frac{1-\sqrt{1-4x}}{2x}.$$ This is the generating function of Catalan numbers.  Then it
is easy to see that $f_t = \frac{1}{t+1}\binom{2t}{t}=
\frac{1}{2t+1}\binom{2t+1}{t}$.  \end{proof}

\vspace{1ex}

Suppose that $p=2$ in Theorem \ref{main}. Then it is easy to see
that Theorem \ref{main2} is a direct corollary of Theorem
\ref{main}.

\noindent\textbf{Proof of Theorem \ref{main2}.}  Let $p=2$ in
Theorem \ref{main}. We have the generating function of $f_t$ is
$$\sum_{t\geq 0}f_tx^t=\frac{1-x-\sqrt{1-2x-3x^2}}{2x^2}.$$ By $A001006$ in \cite{sloane} we
know this is the generating function of Motzkin numbers. It is
well-known that  Bernhart \cite{bern} proved that the $t-$th Motzkin
number equals to $\sum_{0\leq k\leq [\frac{t}{2}]
}\frac{1}{k+1}\binom{t}{2k}\binom{2k}{k}$. We finish the proof.
\hfill $\square$


\section{Acknowledgements}
The author appreciates  Prof. P. O. Dehaye's   encouragement and
help. I would also like to thank Prof. C. Krattenthaler for the
useful comments and thank Prof. W. J. Keith and Prof. R. Nath for
making me aware of \cite{KN}. The author is supported  by
Forschungskredit of the University of Zurich, grant no. [FK-14-093].


\begin{thebibliography}{1}





\bibitem{tamd}
T. Amdeberhan, Theorems, problems and conjectures, $2013.$ Published
electronically at
\emph{www.math.tulane.edu$/\sim$tamdeberhan$/$conjectures.pdf}.

\bibitem{and}
J. Anderson, Partitions which are simultaneously $t_1$- and
$t_2$-core, Disc. Math. $248(2002),\ 237-243$.




\bibitem{berge} C. Berge, Principles of combinatorics, Mathematics in Science and Engineering Vol. $72$, Academic
Press, New York, $1971$.

\bibitem{bern}
F. R. Bernhart, Catalan, Motzkin, and Riordan numbers, Discrete
Math. $204 (1999),\  73-112.$




\bibitem{ford}
B. Ford, H. Mai, and L. Sze, Self-conjugate simultaneous $p$- and
$q$-core partitions and blocks of $A_n$, J. Number Theory
$129(4)(2009),\ 858-865.$

\bibitem{stanton2}
F. Garvan, D. Kim, and D. Stanton, Cranks and $t$-cores, Inv. Math.
$101 (1990),\ 1-17.$

\bibitem{gran}
A. Granville and K. Ono, Defect zero $p$-blocks for finite simple
groups, Trans. Amer. Math. Soc. $348 (1996),\ 331-347.$





\bibitem{jk}
G. James, A. Kerber, The representation theory of the symmetric
group, Addison-Wesley Publishing Company, Reading, MA, $1981$.


\bibitem{KN}
W. J. Keith and R. Nath,  Partitions with prescribed hooksets, J.
Comb. Number Theory $3(1) (2011),\  39-50$.


\bibitem{ols}
J. Olsson and D. Stanton, Block inclusions and cores of partitions,
Aequationes Math. $74(1-2)(2007),\ 90-110$.

\bibitem{sloane}
 N. J. A. Sloane, The On-Line Encyclopedia of
Integer Sequences, $2014.$ Published electronically at
\emph{https$://$oeis.org}.

\bibitem{stanley}
R. P. Stanley and F. Zanello, The Catalan case of Armstrong's
conjectures on simultaneous core partitions, arXiv$:1312.4352$.

\bibitem{stanton}
D. Stanton, Open positivity conjectures for integer partitions,
Trends Math. $2 (1999),\ 19 - 25$.
















\end{thebibliography}
\end{document}